\documentclass{amsart}
\usepackage{amsfonts}
\usepackage{bbm}
\usepackage{mathrsfs}
\usepackage{amsmath}
\usepackage{bigdelim}
\usepackage{multirow}
\usepackage{amssymb}
\usepackage{indentfirst}
\usepackage{CJK}
\usepackage{fancyhdr}
\usepackage{amscd,amssymb,amsmath,graphicx,verbatim,xy,cases}
\usepackage[TS1,OT1,T1]{fontenc}

\newtheorem{theorem}{Theorem}[section]
\newtheorem{lemma}[theorem]{Lemma}
\newtheorem{corollary}[theorem]{Corollary}
\newtheorem{proposition}[theorem]{Proposition}

\theoremstyle{definition}
\newtheorem{definition}[theorem]{Definition}
\newtheorem{example}[theorem]{Example}

\theoremstyle{remark}
\newtheorem{remark}[theorem]{Remark}
\numberwithin{equation}{section}

\begin{document}

\title[Some invariants of $U(1,1;\mathbb{H})$ and diagonalization]{Some invariants of $U(1,1;\mathbb{H})$ and diagonalization}

\author{Cailing Yao}
\address{Cailing Yao, School of Mathematics, Jilin University, 130012, Changchun, P. R. China}
\email{1290279144@qq.com}

\author{Bingzhe Hou$^*$}
\address{Bingzhe Hou, School of Mathematics, Jilin University, 130012, Changchun, P. R. China}
\email{houbz@jlu.edu.cn}
	
\author{Xiaoqi Feng}
\address{Xiaoqi Feng, School of Mathematics, Jilin University, 130012, Changchun, P. R. China}
\email{fengxiaoqi2011@qq.com}
\date{}
\thanks{*Corresponding author.}
\begin{abstract}
Denote by $\mathbb{H}$ the set of all quaternions. We are interested in the group $U(1,1;\mathbb{H})$, which is a subgroup of $2\times 2$ quaternionic matrix group and is sometimes called $Sp(1,1)$. As well known, $U(1,1;\mathbb{H})$ corresponds to the quaternionic M\"{o}bius  transformations on the unit ball in $\mathbb{H}$. In this article, some similar invariants on $U(1,1;\mathbb{H})$ are discussed. Our main result shows that each matrix $T\in U(1,1;\mathbb{H})$, which corresponds to an elliptic quaternionic M\"{o}bius  transformation $g_T(z)$, could be $U(1,1;\mathbb{H})$-similar to a diagonal matrix.
\end{abstract}
\maketitle

\section{Introduction}
Denote by $\mathbb{H}$ the set of all quaternions. Each element of $\mathbb{H}$ has the form $z=z_{0}+\mathbf{i}z_{1}+\mathbf{j}z_{2}+\mathbf{k}z_{3}$ where $z_{i}\in\mathbb{R}$ and $\mathbf{i}^{2}=\mathbf{j}^{2}=\mathbf{k}^{2}=\mathbf{ijk}=-1$. Denote by $Re(z)=z_{0}$ and $Im(z)=\mathbf{i}z_{1}+\mathbf{j}z_{2}+\mathbf{k}z_{3}$ the real and imaginary part of $z$, respectively. Moreover, let the quaternionic conjugate of $z$ be $\overline{z}=z_{0}-z_{1}\mathbf{i}-z_{2}\mathbf{j}-z_{3}\mathbf{k}$, and let the modulus of $z$ be
$\mid z\mid=\sqrt{\overline{z}z}=\sqrt{z_{0}^{2}+z_{1}^{2}+z_{2}^{2}+z_{3}^{2}}$. The family of quaternions plays a role in quantum physics \cite{Adl94}, \cite{Fin62} and \cite{Fin79}. As analogues of complex matrices, quaternionic matrices have gained much attention, see \cite{Wolf36}, \cite{Lee49}, \cite{JLB51} and a good survey by Zhang \cite{FZZ1997}.

Denote by $M_{n}(\mathbb{H})$ the set of all $n\times n$ matrices of quaternions. Moreover, denote by
$GL_{n}(\mathbb{H})$ the set of all invertible matrices in $M_{n}(\mathbb{H})$, and denote by
$Sp_{n}(\mathbb{H})$ the set of all quaternionic unitary matrices in $M_{n}(\mathbb{H})$, i.e.,
\[
Sp_{n}(\mathbb{H})\triangleq\{U\in M_{n}(\mathbb{H}); ~U^*U=UU^*=\mathbf{I}_n\},
\]
where $U^*$ is the quaternionic adjoint matrix of $U$ and $\mathbf{I}_n$ is the $n\times n$ identity matrix. Let $\mathfrak{G}$ be a subgroup of $M_{n}(\mathbb{H})$ equipped with matrix multiplication. For two matrices $A, B\in M_{n}(\mathbb{H})$, we call that $A$ is $\mathfrak{G}$-similar to $B$ if there exists a matrix $G\in\mathfrak{G}$ such that $G^{-1}AG=B$. The problem of similar classification is a key problem in matrix theory. The usual similarity is the $GL_{n}(\mathbb{H})$-similarity. Wolf \cite{Wolf36} studied similarity of quaternionic matrices as early as 1936, and then Wiegmann obtained an analogue of Jordan canonical decomposition theorem for quaternionic matrices. In particular, two quaternions $z$ and $\omega$ are similar, i.e., there exists a nonzero $q\in $ $\mathbb{H}$ such that $z=q\omega q^{-1}$ if and only if $\mid z\mid=\mid \omega\mid$ and $z_{0}=\omega_{0}$ (see \cite{JLB51}). The $Sp_{n}(\mathbb{H})$-similarity is often said to be the quaternionic unitary equivalence, which was studied in \cite{DB2008}.

Of course, similar invariants are important in the study of similar classification. For $A=[a_{ij}]\in M_{n}(\mathbb{H})$, the trace $Tr$ of $A$ introduced in \cite{DB2008} is,
$$
Tr(A):=\sum_{i=1}^{n}(a_{ii}+\overline{a_{ii}})=2Re(\sum_{i=1}^{n}a_{ii}).
$$
One should notice the distinction between the usual matrix trace and the $Tr$ defined above. In fact, it is not difficult to see that the $Tr$ is a similar invariant while the usual matrix trace is not (\cite{FZZ1997}). Some identities on $Tr$ are also shown in \cite{DB2008}.

Eigenvalues or spectra is a classical concept in the study of matrix theory. Let $V$ be a two-side quaternionic vector space. A map $T:V\rightarrow V$ is said to be a right linear operator if for any $\mathbf{u},\mathbf{v}\in V$ and any $s\in\mathbb{H}$
\[
T(\mathbf{u}+\mathbf{v})=T(\mathbf{u})+T(\mathbf{v}),\ \ \ \ T(\mathbf{u}s)=T(\mathbf{u})s.
\]
Similarly, a map $T:V\rightarrow V$ is said to be a left linear operator if
\begin{center}
$T(\mathbf{u}+\mathbf{v})=T(\textbf{u})+T(\mathbf{v}),\ \ T(s\mathbf{u})=sT(\mathbf{u})$,
\end{center}
for all $s\in$ $\mathbb{H}$, $\mathbf{u},\mathbf{v}\in V$. We use $\mathcal{B}^{R}(V)$ ($\mathcal{B}^{L}(V)$) to denote the set of all right (left) linear bounded operators on the two-side vector space $V$.  In this paper, we consider the right linear bounded operators only, because the study on $\mathcal{B}^{R}(V)$ is the same as the study on $\mathcal{B}^{L}(V)$.

\begin{definition}\label{RS}
Let $T:V\rightarrow V$ be a right linear quaternionic operator on $V$. The set of all right eigenvalues of $T$, denoted by $\sigma_{Rp}(T)$, is defined by
$$
\sigma_{Rp}(T)=\{s\in \mathbb{H}; \text{there exists a nonzero vector}~\mathbf{v}~\text{such that}~ T\mathbf{v}-\mathbf{v}s=\mathbf{0}\}.
$$
The right spectrum of $T$, denoted by $\sigma_{R}(T)$, is defined by
$$
\sigma_{R}(T)=\{s\in \mathbb{H}; T-\mathbf{I}s ~\text{is not invertible}\},
$$
where $\mathbf{I}$ denotes the identity operator on $V$, and the notation $\mathbf{I}s$ means that $(\mathbf{I}s)(\mathbf{v})=\mathbf{v}s$ for any $\mathbf{v}\in V$.
\end{definition}

\begin{definition}\label{LS}
Let $T:V\rightarrow V$ be a left linear quaternionic operator on $V$. The set of all left eigenvalues of $T$, denoted by $\sigma_{Lp}(T)$, is defined by
$$
\sigma_{Lp}(T)=\{s\in \mathbb{H}; \text{there exists a nonzero vector}~\mathbf{v}~\text{such that}~ T\mathbf{v}-s\mathbf{v}=\mathbf{0}\}.
$$
The left spectrum of $T$, denoted by $\sigma_{L}(T)$, is defined by
$$
\sigma_{L}(T)=\{s\in \mathbb{H}; T-s\mathbf{I} ~\text{is not invertible}\},
$$
where the notation $s\mathbf{I}$ means that $(s\mathbf{I})(\mathbf{v})=s\mathbf{v}$ for any $\mathbf{v}\in V$.
\end{definition}

In the study of noncommutative functional calculus, the $S$-spectrum for a right linear quaternionic operator was introduced, we refer to \cite{NFC}.
\begin{definition}\label{SS}
Let $T\in $ $\mathcal{B}^{R}(V)$. The $S$-spectrum of $T$, denoted by $\sigma_{S}(T)$, is defined by
$$
\sigma_{S}(T)=\{ s\in \mathbb{H}; T^{2}-2s_{0}T+\mid s\mid^{2}\mathbf{I} \ \text{is  not  invertible}\}.
$$
Moreover, we define the $S$-point spectrum of $T$ by
$$
\sigma_{Sp}(T)=\left\{
s\in \mathbb{H}; \begin{array}{cc}\text{there exists a nonzero vector}~\mathbf{v}~\text{such} \\
\text{that}~(T^{2}-2s_{0}T+\mid s\mid^{2}\mathbf{I})\mathbf{v}=\mathbf{0}
\end{array}
\right\}.
$$
\end{definition}
Obviously, if $T$ is a right linear quaternionic operator (in fact, it is a matrix) on a finite dimensional two-side quaternionic vector space $V$, then
\[
\sigma_{Rp}(T)=\sigma_{R}(T), \ \ \ \sigma_{Lp}(T)=\sigma_{L}(T) \ \ \ \text{and} \ \ \  \sigma_{Sp}(T)=\sigma_{S}(T).
\]
For the case of infinite dimension, we refer to \cite{NFC} and \cite{HT}.

The study on some subgroups of the quaternionic matrix group $M_{n}(\mathbb{H})$ and the related similarity problems is an interesting and important topic. In the present paper, we focus on the subgroup $U(1,1;\mathbb{H})$ of $M_{2}(\mathbb{H})$, which is also called $Sp(1,1)$ somewhere. Let us introduce this subgroup firstly, more details and higher dimension cases can be seen in\cite{IKJRP2003}, \cite{WJX04} and \cite{CWS2016}.

Let $\mathbb{H}^{1,1}$ be the left vector space of dimension $2$ over $\mathbb{H}$ with the unitary structure defined by the Hermitian form
$$
\langle \mathbf{z},\mathbf{w} \rangle=\mathbf{w}^{\ast}J\mathbf{z}=\overline{w}z-\overline{w'}z',
$$
where $\mathbf{z}$ and $\mathbf{w}$ are the column vectors in $\mathbb{H}^{1,1}$ with entries $(z,z')$ and $(w,w')$ respectively, $\mathbf{w}^{\ast}$ denotes the conjugate transpose of $\mathbf{w}$ and $J$ is the Hermitian matrix $J=\begin{bmatrix}
1 & 0 \\
0 & -1
\end{bmatrix}$.
An automorphism $A$ on $\mathbb{H}^{1,1}$ is said to be a $J$-unitary transformation, if the automorphism $A$ satisfies $\langle A\mathbf{z}, A\mathbf{w} \rangle=\langle \mathbf{z},\mathbf{w} \rangle$ for all $\mathbf{z}$ and $\mathbf{w}$ in $\mathbb{H}^{1,1}$, that is, $A\in GL_{2}(\mathbb{H})$ and $A^*JA=J$. Denote the set of all $J$-unitary transformations by $U(1,1;\mathbb{H})$. Following from \cite{WJX04}, one can see that a matrix $T=\begin{bmatrix}
a & b \\
c & d
\end{bmatrix}$ is an element of $U(1,1;\mathbb{H})$ if and only if
\begin{equation}\label{1.1}
\mid a\mid=\mid d\mid,\ \mid b\mid=\mid c\mid,\ \mid a\mid^{2}-\mid c\mid^{2}=1,\ \overline{a}b=\overline{c}d,\ a\overline{c}=b\overline{d}.
\end{equation}
In addition, the set of all complex matrices satisfying (\ref{1.1}) is denoted by $U(1,1;\mathbb{C})$, equivalently,
\[
U(1,1;\mathbb{C})\triangleq\{A=\begin{bmatrix}
a & b \\
c & d
\end{bmatrix}\in M_2(\mathbb{C}); ~b=\bar{c},~d=\bar{a}~ \text{and}~ |a|^2-|c|^2=1\}.
\]

One reason we pay attention to $U(1,1;\mathbb{H})$ is that it is closely related to quaternionic M\"{o}bius transformations on the unit ball $\mathbb{B}\triangleq\{q\in\mathbb{H}; |q|<1\}$. Generally, M\"{o}bius transformations are the orientation preserving isometries of $n$-dimensional hyperbolic space. Ahlfors \cite{Ahlfors85} related  M\"{o}bius transformations to two-by-two matrices whose entries lie in a Clifford algebra. Then, in the study of M\"{o}bius transformations, the corresponding matrix group and hyperbolic metric are closely interconnected, see \cite{Ahlfors85F}, \cite{CW98}, \cite{G01}, \cite{IKJRP2003}, \cite{CW07}, \cite{BC09} and \cite{CWS2016} for example. In particular, W. Cao, J. Parker and X. Wang \cite{WJX04} studied the classification of quaternionic M\"{o}bius transformations on $\mathbb{B}$. They proved that every matrix $T=\begin{bmatrix}
a & b \\
c & d
\end{bmatrix}\in U(1,1;\mathbb{H})$ corresponds to a quaternionic M\"{o}bius transformation $g_T(z)=(az+b)(cz+d)^{-1}$ on the unit ball $\mathbb{B}$ and every quaternionic M\"{o}bius transformation on the unit ball $\mathbb{B}$ could be written as $g_T(z)=(az+b)(cz+d)^{-1}$ for some $T\in U(1,1;\mathbb{H})$, which is an analogy of the classical M\"{o}bius transformation group on the unit disk $\mathbb{D}\triangleq\{z\in\mathbb{C}; |z|<1\}$ is isomorphic to the group $PU(1,1;\mathbb{C})\triangleq U(1,1;\mathbb{C})/\{\pm \mathbf{I}_2\}$.
In fact,
\[
T=\begin{bmatrix}
a & b \\
c & d
\end{bmatrix}
\longrightarrow
g_T(z)=(az+b)(cz+d)^{-1}
\]
induces a group homomorphism from $U(1,1;\mathbb{H})$ to the group of all quaternionic M\"{o}bius transformations on $\mathbb{B}$ with kernel $\{\mathbf{I}_2, -\mathbf{I}_2\}$. In \cite{WJX04}, the notion $\triangle$ defined by
\begin{equation}\label{1.2}
\triangle=\mid Im(\overline{c}^{-1}b\overline{d}+d)\mid^{2}-\mid\overline{c}^{-1}b-1\mid^{2}, \ \ \text{if} \ b\neq \overline{c}\neq 0,
\end{equation}
plays an important role. Moreover, similar to the classification of classical M\"{o}bius transformations on $\mathbb{D}$ (see \cite{EAN1968} for example), Cao et.al.classified the quaternionic M\"{o}bius transformations $(az+b)(cz+d)^{-1}$ by $a,~b,~c,~d$. For a quaternionic M\"{o}bius transformation $g=(az+b)(cz+d)^{-1}$ preserving $\mathbb{B}$,

\noindent I. $b=\overline{c}=0$, then

(I.1) \ if $a_{0}=d_{0}$,  $g$ is simple elliptic;

(I.2) \ if $a_{0}\neq d_{0}$,  $g$ is compound elliptic;

\noindent II.  $b=\overline{c}\neq 0$, then

(II.1) \  if $d_{0}^{2}<1$,  $g$ is simple elliptic;

(II.2) \  if $d_{0}^{2}=1$,  $g$ is simple parabolic;

(II.3) \  if $d_{0}^{2}>1$,  $g$ is simple loxodromic;

\noindent III.  $b\neq\overline{c}\neq 0$, then

(III.1) \  if $\triangle<0$,  $g$ is compound elliptic;

(III.2) \  if $\triangle=0$,  $g$ is compound parabolic;

(III.3) \ if $\triangle>0$,  $g$ is compound loxodromic.

Furthermore, Zhou and Cao obtained the sufficient and necessary conditions for elliptic transformations and parabolic transformations from the view of geometry \cite{ZC2015}.

In this paper, we are interested in the similarity problem on $U(1,1;\mathbb{H})$, especially the $U(1,1;\mathbb{H})$-similarity problem, because $U(1,1;\mathbb{H})$-similar transformation corresponds to a composition of quaternionic M\"{o}bius transformations. In section $2$, some similar invariants on $U(1,1;\mathbb{H})$ are studied. In section $3$, we show that each matrix $T\in U(1,1;\mathbb{H})$, which corresponds to an elliptic quaternionic M\"{o}bius  transformation $g_T(z)$, could be $U(1,1;\mathbb{H})$-similar to a diagonal matrix. Notice that $U(1,1;\mathbb{H})$-similarity problem is more complicated because of the lack of commutativity of quaternion multiplication and the restriction of the similar transformation matrix in $U(1,1;\mathbb{H})$.

\section{Some invariants in $U(1,1;\mathbb{H})$}

In this section, we study some similar invariants in $U(1,1;\mathbb{H})$.

\subsection{Trace and $\Delta$}

Recall that For any $T=\begin{bmatrix}
a & b \\
c & d
\end{bmatrix}\in M_2(\mathbb{H})$, the trace $Tr$ of $T$ is $Tr(T)=2(a_{0}+d_{0})$. The trace $Tr$ has been well studied by Dokovi\'{c} and Smith \cite{DB2008}. In particular, $Tr$ is a similar invariant.

In \cite{WJX04}, the notation
\[
\triangle=Im(\overline{c}^{-1}b\overline{d}+d)\mid^{2}-\mid\overline{c}^{-1}b-1\mid^{2}
\]
for $T=\begin{bmatrix}
a & b \\
c & d
\end{bmatrix}\in U(1,1;\mathbb{H})$ with $b\neq \overline{c}\neq 0$ was introduced to study the classification of quaternionic M\"{o}bius transformations. In the present paper, we will extend this notation to the whole $U(1,1;\mathbb{H})$ by another form. For any $T=\begin{bmatrix}
a & b \\
c & d
\end{bmatrix}\in U(1,1;\mathbb{H})$, define
\[
\triangle(T)=\mid b-\overline{c}\mid^{2}-(a_{0}-d_{0})^{2}.
\]
\begin{lemma}
For any $T=\begin{bmatrix}
a & b \\
c & d
\end{bmatrix}\in U(1,1;\mathbb{H})$ with $b\neq \overline{c}\neq 0$, $\triangle(T)$ coincides with $\triangle$.
\end{lemma}
\begin{proof}
\begin{align*}
\triangle(T)&=\mid b-\overline{c}\mid^{2}-(a_{0}-d_{0})^{2}\\
&=\mid \overline{c}(\overline{c}^{-1}b-1)\mid^{2}-\mid Re(a-d)\mid^{2}\\
&=\mid \overline{c}^{-1}b-1\mid^{2}(\mid d\mid^{2}-1)-\mid Re(\overline{c}^{-1}b\overline{d}-\overline{d})\mid^{2}\\
&=\mid Im(\overline{c}^{-1}b-1)\overline{d}\mid^{2}-\mid \overline{c}^{-1}b-1\mid^{2}\\
&=\mid Im(\overline{c}^{-1}b\overline{d}+d)\mid^{2}-\mid\overline{c}^{-1}b-1\mid^{2}\\
&=\triangle.
\end{align*}
\end{proof}

For convenience, we will use $\triangle(T)$ instead of $\triangle$. Moreover, it is not difficult to find that
\begin{enumerate}
\item $\triangle(T)=0$, if $b=\overline{c}=0,\ a_{0}=d_{0}$,
\item $\triangle(T)<0$, if $b=\overline{c}=0,\ a_{0}\neq d_{0}$,
\item $\triangle(T)=0$, if $b=\overline{c}\neq 0$,
\item $\triangle(T)=\triangle$, if $b\neq \overline{c}\neq 0$.
\end{enumerate}

To consider the elements in $U(1,1;\mathbb{H})$, one can see that $\triangle(T)$ is decided by $Tr(T)$ and $Tr(T^2)$.
\begin{theorem}\label{triT}
For any $T=\begin{bmatrix}
a & b \\
c & d
\end{bmatrix}\in U(1,1;\mathbb{H})$, we have
\[
\triangle(T)=\frac{1}{4}(Tr(T))^{2}-\frac{1}{2}Tr(T^{2})-2.
\]
Furthermore, $\triangle(T)$ is a similar invariant in $U(1,1;\mathbb{H})$.
\end{theorem}
\begin{proof}
Computing straightly, we have
\[
T^{2}=\begin{bmatrix}
a^{2}+bc & ab+bd \\
ca+dc & cb+d^{2}
\end{bmatrix}.
\]
Then
\begin{align*}
Tr(T^{2})=&2Re(a^{2}+bc+cb+d^{2})\\
=& 4a_{0}^{2}+4d_{0}^{2}-4-2(\mid b\mid^{2}+\mid c\mid^{2}-2Re(bc)) \\
=& 4a_{0}^{2}+4d_{0}^{2}-4-2\triangle(T)-2(a_{0}-d_{0})^{2}\\
=& 2(a_{0}+d_{0})^{2}-4-2\triangle(T),
\end{align*}
and consequently,
\begin{equation}\label{Tr}
\triangle(T)=\frac{1}{4}(Tr(T))^{2}-\frac{1}{2}Tr(T^{2})-2.
\end{equation}
Since $Tr(T)$ and $Tr(T^{2})$ are both similar invariants, then $\triangle(T)$ is also a similar invariant in $U(1,1;\mathbb{H})$.
\end{proof}

We could also obtain some identities with respect to $\triangle(T^{n})$, for some positive integer $n$.

\begin{proposition}\label{triT2}
Let $T=\begin{bmatrix}
a & b \\
c & d
\end{bmatrix}\in U(1,1;\mathbb{H})$. Then,
\[
\triangle(T^{2})=(Tr(T))^{2}\cdot\triangle(T).
\]
\end{proposition}
\begin{proof}
Following from
\[
T^{2}=\begin{bmatrix}
a^{2}+bc & ab+bd \\
ca+dc & cb+d^{2}
\end{bmatrix}, \ \ \ a\overline{c}=b\overline{d} \ \ \ \text{and} \ \ \ \overline{a}b=\overline{c}d,
\]
one can see that
\begin{align*}
\triangle(T^{2})=& \mid ab+bd-\overline{a}\cdot\overline{c}-\overline{c}\overline{d}\mid^{2}-(2a_{0}^{2}-2d_{0}^{2})^{2} \\
=&\mid ab-a\overline{c}+a\overline{c}+bd-\overline{c}d+\overline{c}d+\overline{a}b-\overline{a}b-\overline{a}\cdot\overline{c}+b\overline{d}-b\overline{d}-\overline{c}\overline{d} \mid^{2} \\
&-4(a_{0}^{2}-d_{0}^{2})^{2} \\
=& \mid (a+\overline{a})(b-\overline{c})+(b-\overline{c})(d+\overline{d}) \mid^{2}-4(a_{0}+d_{0})^{2}(a_{0}-d_{0})^{2} \\
=&4(a_{0}+d_{0})^{2}(\mid b-\overline{c} \mid^{2}-(a_{0}-d_{0})^{2})\\
=&(Tr(T))^{2}\cdot\triangle(T).
\end{align*}
\end{proof}

\begin{corollary}
Let $T=\begin{bmatrix}
a & b \\
c & d
\end{bmatrix}\in U(1,1;\mathbb{H})$ with $Tr(T)=0$. Then, $T^2$ corresponds to a simple quaternionic M\"{o}bius transformation $g_{T^2}$.
\end{corollary}

\begin{proof}
Denote
\[
T^{2}=\begin{bmatrix}
T^{(2)}_{11} & T^{(2)}_{12} \\
T^{(2)}_{21} & T^{(2)}_{22}
\end{bmatrix}=\begin{bmatrix}
a^{2}+bc & ab+bd \\
ca+dc & cb+d^{2}
\end{bmatrix}.
\]
One can see that
\begin{align*}
T^{(2)}_{12}-\overline{T^{(2)}_{21}}=&ab+bd-\overline{(ca+dc)} \\
=&ab+bd-\bar{a}\bar{c}-\bar{c}\bar{d}   \\
=&ab+\bar{a}b-\bar{a}b+bd+b\bar{d}-b\bar{d}-\bar{a}\bar{c}-a\bar{c}+a\bar{c}-\bar{c}\bar{d}-\bar{c}d+\bar{c}d   \\
=&2a_0b+2d_0b-2a_0\bar{c}-2d_0\bar{c}-\bar{a}b-b\bar{d}b+a\bar{c}+\bar{c}d \\
=&2(a_0+d_0)(b-\bar{c}) \\
=&0.
\end{align*}
and by Proposition \ref{triT2} and $Tr(T)=0$,
\[
\triangle(T^{2})=(Tr(T))^{2}\cdot\triangle(T)=0,
\]
and consequently $Re(T^{(2)}_{11})=Re(T^{(2)}_{22})$.

Therefore,
\begin{enumerate}
\item if $T^{(2)}_{12}=\overline{T^{(2)}_{21}}\neq0$, then $g_{T^2}$ is a simple quaternionic M\"{o}bius transformation;

\item if $T^{(2)}_{12}=\overline{T^{(2)}_{21}}=0$, then it follows from $Re(T^{(2)}_{11})=Re(T^{(2)}_{22})$ that $g_{T^2}$ is a simple elliptic quaternionic M\"{o}bius transformation.
\end{enumerate}
The proof is finished.
\end{proof}

\begin{corollary}
Let $T=\begin{bmatrix}
a & b \\
c & d
\end{bmatrix}\in U(1,1;\mathbb{H})$ with $Tr(T)=0$. Then, $Tr(T^{2})\leq 4$ and the equality holds if and only if $d\in\mathbb{R}$.
\end{corollary}
\begin{proof}
Let us prove this conclusion in the cases $b=\overline{c}$ and $b\neq\overline{c}$, respectively.

{\bf 1.} Suppose that $b=\overline{c}$.

If $\triangle(T)=0$, then it follows from (\ref{Tr}) that $Tr(T^{2})=-4< 4$.

If $\triangle(T)<0$, then $\ a_{0}=-d_{0}\neq 0$ and consequently $\triangle(T)=-4d_{0}^{2}$. Notice that in this case, $b=\overline{c}=0$ and consequently $|a|=|d|=1$. Hence
\[
Tr(T^{2})=8d_{0}^{2}-4=8(d_{0}^{2}-1)+4\leq 4
\]
and the equality holds if and only if $d=\pm 1$.

{\bf 2.} Suppose that  $b\neq\overline{c}$.

The assumption $Tr(T)=0$ means that $a_{0}+d_{0}=Re(\overline{c}^{-1}b\overline{d})+d_{0}=0$. Let
\[
\overline{c}^{-1}b=R_{t}+\sqrt{1-R_{t}^{2}}I_{t} \ \ \ \text{and} \ \  \ \overline{d}=d_{0}+\sqrt{\mid d\mid^{2}-d_{0}^{2}}I_{d},
\]
where $I_{t}$ and $I_{d}$ are the imaginary units of $\overline{c}^{-1}b$ and $\overline{d}$, respectively. Then
$$
Re(\overline{c}^{-1}b\overline{d})+d_{0}=R_{t}d_{0}+\sqrt{1-R_{t}^{2}} \sqrt{\mid d\mid^{2}-d_{0}^{2}} Re(I_{t}I_{d})+d_{0}=0,
$$
that is
$$
(R_{t}+1)d_{0}=-\sqrt{1-R_{t}^{2}} \sqrt{\mid d\mid^{2}-d_{0}^{2}} Re(I_{t}I_{d}).
$$
Then
$$
\mid(R_{t}+1)d_{0}\mid=\sqrt{1-R_{t}^{2}} \sqrt{\mid d\mid^{2}-d_{0}^{2}} \mid Re(I_{t}I_{d})\mid\leq \sqrt{1-R_{t}^{2}} \sqrt{\mid d\mid^{2}-d_{0}^{2}},
$$
and consequently,
\[
\mid d\mid^{2}R_{t}^{2}+2d_{0}^{2}R_{t}+2d_{0}^{2}-\mid d\mid^{2}\leq 0.
\]
This implies
$$
R_{t}\leq \frac{\mid d\mid^{2}-2d_{0}^{2}}{\mid d\mid^{2}}.
$$
By the definition of $\triangle(T)$, we have
\[
\triangle(T)=2\mid c\mid^{2}(1-R_{t})-4d_{0}^{2}\geq \frac{4(\mid d\mid^{2}-1)d_{0}^{2}}{\mid d\mid^{2}}-4d_{0}^{2}=-\frac{4\mid d_0\mid^{2}}{\mid d\mid^{2}}.
\]
Furthermore, one can obtain that $\triangle(T)\geq -4$ and the equality holds if and only if $d$ is a real number. Therefore,
\[
Tr(T^2)=\frac{1}{2}(Tr(T))^2-4-2\triangle(T)\le 4,
\]
and the equality holds if and only if $d$ is a real number.

In a word, we always have $Tr(T^{2})\leq 4$ and the equality holds if and only if $d\in\mathbb{R}$. The proof is finished.
\end{proof}

\begin{proposition}\label{triT3}
Let $T=\begin{bmatrix}
a & b \\
c & d
\end{bmatrix}\in U(1,1;\mathbb{H})$. Then,
\[
\triangle(T^{3})=\left(\frac{1}{2}(Tr(T))^{2}+\frac{1}{2}Tr(T^{2})-1\right)^{2}\cdot\triangle(T).
\]
\end{proposition}
\begin{proof}
By a simple calculation, one can see
$$
T^{3}=\begin{bmatrix}
a^{3}+bca+abc+bdc & a^{2}b+bd^{2}+bcb+abd \\
ca^{2}+d^{2}c+dca+cbc & cbd+cab+dcb+d^{3}
\end{bmatrix}.
$$
Then,
\begin{align*}
&\mid a^{2}b+bd^{2}+bcb+abd-\overline{a}^{2}\overline{c}-\overline{c}\overline{d}^{2}-\overline{a}\cdot\overline{c}\overline{d}-\overline{c}\overline{b}\overline{c} \mid^{2}\\
=& |(a^{2}+\overline{a}^{2})(b-\overline{c})+(b-\overline{c})(d^{2}+\overline{d}^{2})+(bc+\overline{c}\overline{b})(b-\overline{c}) \\ &+(a+\overline{a})(b-\overline{c})(d+\overline{d})-a(b-\overline{c})\overline{d}-\overline{a}(b-\overline{c})d+ab\overline{d}+\overline{a}bd \\
&+b\mid c\mid^{2}+b\mid d\mid^{2}-\overline{a}\cdot\overline{c}d-a\overline{c}\overline{d}-\overline{c}\mid d\mid^{2}-\overline{c}\mid b\mid^{2} |^{2}\\
=&(2Re(a^{2})+2Re(d^{2})+2Re(bc) +4a_{0}d_{0}+\mid b\mid^{2}+\mid d\mid^{2})^{2}\cdot\mid b-\overline{c}\mid^{2}\\
=&\left(\frac{1}{2}(Tr(T))^{2}+\frac{1}{2}Tr(T^{2})-1\right)^{2}\cdot \mid b-\overline{c}\mid^{2},
\end{align*}
and
\begin{align*}
&(Re(a^{3}+bca+abc+bdc-d^{3}-cbd-cab-dcb))^{2}\\
=&(Re((a^{2}+bc)(a+\overline{a})-(d+\overline{d})(d^{2}+cb)-(a^{2}+bc)\overline{a}+\overline{d}(d^{2}+cb)))^{2} \\
=&(2Re(a^{2})+2Re(d^{2})+2Re(bc) +4a_{0}d_{0}+\mid b\mid^{2}+\mid d\mid^{2})^{2}\cdot (a_{0}-d_{0})^{2}\\
=&\left(\frac{1}{2}(Tr(T))^{2}+\frac{1}{2}Tr(T^{2})-1\right)^{2}\cdot (a_{0}-d_{0})^{2}.
\end{align*}
Hence
\[
\triangle(T^{3})=\left(\frac{1}{2}(Tr(T))^{2}+\frac{1}{2}Tr(T^{2})-1\right)^{2}\cdot\triangle(T).
\]
\end{proof}
According to Theorem \ref{triT}, Proposition \ref{triT2} and Proposition \ref{triT3}, one can obtain the following identities for $\triangle(T^{6})$.
\begin{corollary}\label{triT6}
Let $T=\begin{bmatrix}
a & b \\
c & d
\end{bmatrix}\in U(1,1;\mathbb{H})$. Then,
\begin{align*}
\triangle(T^{6})=
&\left(\frac{1}{2}(Tr(T^{2}))^{2}+\frac{1}{2}Tr(T^{4})-1\right)^{2}\cdot (Tr(T))^{2}\cdot \triangle(T)\\
=&\left(\frac{1}{2}(Tr(T))^{2}+\frac{1}{2}Tr(T^{2})-1\right)^{2}\cdot (Tr(T^{3}))^{2}\cdot \triangle(T).
\end{align*}
\end{corollary}

\subsection{Spectra}

Now let us consider the spectra of quaternionic matrices. To simplify our discussion, it is useful of that every quaternionic matrix is similar to a complex matrix \cite{Wie55}. In fact, a Jordan canonical decomposition theorem of quaternionic matrices was obtained by Wiegmann \cite{Wie55}.

\begin{theorem}[Theorem $1$ in \cite{Wie55}]\label{RS1}
Every $n\times n$ matrix with real quaternion elements is similar
under a matrix transformation with real quaternion elements to a matrix in
(complex) Jordan normal form with diagonal elements of the form $a+\mathbf{i}b$, where $a,b\in\mathbb{R}$ and $b\ge 0$.
\end{theorem}

Then, we have the next conclusion according to the theorem above.

\begin{proposition}\label{prop5}
If $s\in\sigma_{S}(T)~(\sigma_{R}(T))$, then $qsq^{-1}\in\sigma_{S}(T)~(\sigma_{R}(T))$ for any nonzero quaternion $q\in\mathbb{H}$. Moreover, for any $T\in M_{n}(\mathbb{H})$, $\sigma_{S}(T)=\sigma_{R}(T)$.
\end{proposition}
\begin{proof}
Suppose that $s\in \sigma_{R}(T)$. Then there is a nonzero vector $\mathbf{v}$ such that $T\mathbf{v}=\textbf{v}s$. For any nonzero quaternion $q\in\mathbb{H}$, we have
\[
T\textbf(\mathbf{v}q^{-1})=\mathbf{v}sq^{-1}=(\mathbf{v}q^{-1})(qsq^{-1}).
\]
Then, $qsq^{-1}\in\sigma_{R}(T)$. On the other hand, since $|s|=|qsq^{-1}|$ and $Re(s)=Re(qsq^{-1})$, it follows from the definition of $\sigma_{S}(T)$ that $qsq^{-1}\in\sigma_{S}(T)$.

Now let us show that $\sigma_{S}(T)=\sigma_{R}(T)$. By the above statements, it suffices to prove $\sigma_{S}(T)\bigcap\mathbb{C}=\sigma_{R}(T)\bigcap\mathbb{C}$.

By Theorem \ref{RS1}, there is a unitary matrix $U$ such that $A=U^{\ast}TU\in M_{n}(\mathbb{C})$. Then,
\[
(A-s)(A-\overline{s})=(A-\overline{s})(A-s)=A^{2}-2s_{0}A+\mid s\mid^{2},\ \text{for  any} \ s \in \mathbb{C}.
\]
Thus, for any $\mathbf{0}\neq \mathbf{v}\in \mathbb{H}^2$,
\begin{equation}\label{shi2}
(A^{2}-2s_{0}A+\mid s\mid^{2})\textbf{v}=(A-s)(A-\overline{s})\mathbf{v}=(A-\overline{s})(A-s)\mathbf{v}.
\end{equation}
It implies that for $s\in\mathbb{C}$, $s\in \sigma_{S}(T)$ if and only if $s\in \sigma_{R}(T)$. This finishes the proof.
\end{proof}

\begin{proposition}
Both $\sigma_{S}$ and $\sigma_{R}$ are similar invariants, but not $\sigma_{L}$.
\end{proposition}
\begin{proof}
Let $T, T'\in M_{n}(\mathbb{H})$. Assume that $T'$ is similar to $T$, that is, $T'=XTX^{-1}$ where $X$ is an invertible matrix in $M_{n}(\mathbb{H})$. If $s\in \sigma_{S}(T)$, then there is a nonzero vector $\textbf{v}$ such that $(T^{2}-2s_{0}T+\mid s\mid^{2}\mathcal{I})\textbf{v}=\textbf{0}$. Set $\textbf{v}'=X\textbf{v}\neq \textbf{0}$. Then it is obvious that
$$
((T')^{2}-2s_{0}T'+\mid s\mid^{2}\mathcal{I})\textbf{v}'=X(T^{2}-2s_{0}T+\mid s\mid^{2}\mathcal{I})\textbf{v}=\textbf{0},
$$
which implies $s\in \sigma_{S}(T')$. Therefore, $\sigma_{S}(T)=\sigma_{S}(T')$. By Proposition \ref{prop5}, we also have $\sigma_{R}(T)=\sigma_{R}(T')$.

However, $\sigma_{L}(T)$ is not a similar invariant. Here we give an example. Let
\[
T=\begin{bmatrix}
\mathbf{i} & 0 \\
0 & \mathbf{j}
\end{bmatrix}\in U(1,1;\mathbb{H}), \ \ \ X=\begin{bmatrix}
1 & 0 \\
0 & -\mathbf{i}
\end{bmatrix}
\]
and
\[
\widetilde{T}=XTX^{-1}=\begin{bmatrix}
\mathbf{i} & 0 \\
0 & -\mathbf{j}
\end{bmatrix}.
\]
Then
\[
\sigma_{L}(T)=\{ \mathbf{i},\ \mathbf{j} \}\neq
\sigma_{L}(\widetilde{T})=\{ \mathbf{i},\ -\mathbf{j} \}.
\]
\end{proof}

\begin{corollary}\label{cor1}
Given any integer $n$. For every $T\in M_{n}(\mathbb{H})$,
\[
\sigma_{R}(T^{n})=\sigma_{S}(T^{n})=\{ s^{n}:s\in\sigma_{S}(T)\}.
\]
However, it is not true for $\sigma_{L}(T^n)$.
\end{corollary}
\begin{proof}
Following from a particular case of the $S$-spectral mapping theorem (Theorem 4.12.5 in \cite{NFC} p155), one can see that $\sigma_{S}(T^{n})=\{ s^{n}:s\in\sigma_{S}(T)\}$. Then, by Proposition \ref{prop5}, the conclusion also holds for $\sigma_{R}(T^{n})$.

Here we show an example in $U(1,1;\mathbb{H})$ such that the conclusion is false for $\sigma_{L}(T^n)$. Let $T=\begin{bmatrix}
\sqrt{2} & \mathbf{i} \\
-\mathbf{i} & \sqrt{2}
\end{bmatrix}\in U(1,1;\mathbb{H})$. Then  $T^{2}=\begin{bmatrix}
3 & 2\sqrt{2}\mathbf{i} \\
-2\sqrt{2}\mathbf{i} & 3
\end{bmatrix}$ and consequently,
$$\sigma_{L}(T)=\{ s\in \mathbb{H}:s_{1}=0,\ (s_{0}-\sqrt{2})^{2}+s_{2}^{2}+s_{3}^{2}=1 \},$$
$$\sigma_{L}(T^{2})=\{ t\in \mathbb{H}:t_{1}=0,\ (t_{0}-3)^{2}+t_{2}^{2}+t_{3}^{2}=8 \}.$$
It is obvious that $\sigma_{L}(T^{2})\neq \{ s^{2}:s\in\sigma_{L}(T)\}$.
\end{proof}

In \cite{WJX04}, Cao et.al. have characterized the right eigenvalues of the elements in $U(1,1;\mathbb{H})$. For $T=\begin{bmatrix}
a & b \\
c & d
\end{bmatrix}\in U(1,1;\mathbb{H})$,

\noindent I. $b=\overline{c}=0$, then

(I.1) \ if $a_{0}=d_{0}$,  $\sigma_{R}(T)=\{ s\in\mathbb{H}:s_{0}=a_{0}=d_{0},\ \mid s\mid=1\}$,

(I.2) \ if $a_{0}\neq d_{0}$,  $\sigma_{R}(T)=\{ s\in\mathbb{H}:s_{0}=a_{0} \ \text{or} \ s_{0}=d_{0},\ \mid s\mid=1\}$.

\noindent II.  $b=\overline{c}\neq 0$, then

(II.1) \  if $d_{0}^{2}=1$,  $\sigma_{R}(T)=\{ s\in\mathbb{H}:s=s_{0}=d_{0}\}$,

(II.2) \  if $d_{0}^{2}>1$,  $\sigma_{R}(T)=\{ s\in\mathbb{H}:s=s_{0}=d_{0}\pm \sqrt{d_{0}^{2}-1}\}$,

(II.3) \  if $d_{0}^{2}<1$,  $\sigma_{R}(T)=\{ s\in\mathbb{H}:s_{0}=d_{0},\ \mid s\mid=1\}$.

\noindent III.  $b\neq\overline{c}\neq 0$, then

(III.1) \  if $\triangle=0$,  $\sigma_{R}(T)=\{ s\in\mathbb{H}:s_{0}=\frac{1}{2}Re(\overline{c}^{-1}b\overline{d}+d), \mid s\mid=1\}$,

(III.2) \  if $\triangle<0$,  $\sigma_{R}(T)=\{ s\in\mathbb{H}:s_{0}=\frac{1}{2}(Re(\overline{c}^{-1}b\overline{d}+d)\pm\sqrt{-\triangle}), \mid s\mid=1\}$,

(III.3) \ if $\triangle>0$,  ~$\sigma_{R}(T)=\{s\in\mathbb{H}:~s_{0}=\frac{1}{2}(Re(\overline{c}^{-1}b\overline{d}+d)+\sqrt{2X-2-\triangle}), \ |s|^{2}=X+\sqrt{X^{2}-1}\} \bigcup \{ s^{\prime}\in\mathbb{H}:s^{\prime}_{0}=\frac{1}{2}(Re(\overline{c}^{-1}b\overline{d}+d)-\sqrt{2X-2-\triangle}),\ \ \mid s^{\prime}\mid^{2}=X-\sqrt{X^{2}-1} \}$, where
\[
X=\frac{(Re(\overline{c}^{-1}b\overline{d}+d))^{2}+\triangle+\sqrt{((Re(\overline{c}^{-1}b\overline{d}+d))^{2}+\triangle-4)^{2}+16\triangle}}{4}>1.
\]

Furthermore, we could  write $\sigma_{R}(T)$ for $T\in U(1,1;\mathbb{H})$ in a unified form via $\triangle(T)$.
\begin{theorem}
Let $T=\begin{bmatrix}
a & b \\
c & d
\end{bmatrix}\in U(1,1;\mathbb{H})$, then
$$
\sigma_{R}(T)=\left\{s\in\mathbb{H}:
\begin{array}{cc}
s_{0}=\frac{1}{2}((a_{0}+d_{0})\pm\sqrt{2X'-2-\triangle(T)}) \\
\text{and}  \ \ \ \mid s\mid^{2}=X'\pm\sqrt{(X')^{2}-1}
\end{array}
\right\},
$$
where
{\small
\[
X'=\frac{(a_{0}+d_{0})^{2}+\triangle(T)+\sqrt{((a_{0}+d_{0})^{2}+\triangle(T)-4)^{2}+8\triangle(T)\left(sgn(\triangle(T))+1\right)}}{4}
\]}
and $sgn$ is the usual sign function on $\mathbb{R}$.
\end{theorem}

\section{Diagonalization of elliptic elements in $U(1,1;\mathbb{H})$}
In this section, we show that every matrix $T\in U(1,1;\mathbb{H})$ corresponding to an elliptic quaternionic M\"{o}bius  transformation $g_{T}$ is $U(1,1;\mathbb{H})$-similar to a diagonal matrix in $U(1,1;\mathbb{H})$. Notice that it is not a trivial result because of the lack of commutativity of quaternion multiplication and the restriction of the similar transformation matrix in $U(1,1;\mathbb{H})$.

\begin{theorem}\label{M1}
Let $T=\begin{bmatrix}
a & b \\
c & d
\end{bmatrix}$ be an matrix in $U(1,1;\mathbb{H})$. If the corresponding quaternionic M\"{o}bius  transformation $g_T(z)=(az+b)(cz+d)^{-1}$ is elliptic, then there exists a matrix $X \in U(1,1;\mathbb{H})$ such that $XTX^{-1}$ is a diagonal matrix.
\end{theorem}
\begin{proof}
Notice that the matrix $T\in U(1,1;\mathbb{H})$ corresponding to an elliptic quaternionic M\"{o}bius transformation satisfies one of the following conditions.
\begin{enumerate}
\item \ $b=\overline{c}=0$.
\item \ $b=\overline{c}\neq 0$ and  $d_{0}^{2}<1$.
\item \ $b\neq \overline{c}\neq 0$ and $\triangle(T)<0$.
\end{enumerate}

\noindent {\bf Case $1$:} $b=\overline{c}=0$.

In this case, the matrix $T$ itself is diagonal.

\noindent {\bf Case $2$:} $b=\overline{c}\neq 0$ and  $d_{0}^{2}<1$.

In this case, we could write
\[
T=\begin{bmatrix}
a & \overline{c} \\
c & d
\end{bmatrix}.
\]
Let
\[
X=\begin{bmatrix}
\overline{q} & 0 \\
0 & 1
\end{bmatrix}\in U(1,1;\mathbb{H}),
\]
where $\overline{q}=\frac{c}{\mid c\mid}=\mid c\mid\cdot\overline{c}^{-1}=q^{-1}$. Then
$$
XTX^{-1}=\begin{bmatrix}
\overline{d} & \mid c\mid \\
\mid c\mid & d
\end{bmatrix},
$$
Furthermore, let $y_{1}$ be the unimodular quaternion such that
\[
y_{1}\overline{d}\overline{y_{1}}=d_{0}+\mathbf{i}\sqrt{1-d_{0}^{2}+\mid c\mid^{2}}
 \]
which is equivalent to
\[
y_{1}d\overline{y_{1}}=d_{0}-\mathbf{i}\sqrt{1-d_{0}^{2}+\mid c\mid^{2}}.
\]
Set
\[
Y=\begin{bmatrix}
y_{1} & 0 \\
0 & y_{1}
\end{bmatrix}\in U(1,1;\mathbb{H}).
\]
Then,
\[
YXTX^{-1}Y^{-1}=\begin{bmatrix}
d_{0}+\mathbf{i}\sqrt{1-d_{0}^{2}+\mid c\mid^{2}} & \mid c\mid \\
\mid c\mid & d_{0}-\mathbf{i}\sqrt{1-d_{0}^{2}+\mid c\mid^{2}}
\end{bmatrix}.
\]

Now, we set
\[
Z=\begin{bmatrix}
k(\sqrt{1-d_{0}^{2}}+\sqrt{1-d_{0}^{2}+\mid c\mid^{2}}) & -k\mid c\mid\mathbf{i} \\
k\mid c\mid\mathbf{i} & k(\sqrt{1-d_{0}^{2}}+\sqrt{1-d_{0}^{2}+\mid c\mid^{2}})
\end{bmatrix},
\]
where $k^{2}=\frac{1}{2-2d_{0}^{2}+2\sqrt{1-d_{0}^{2}}\cdot\sqrt{1-d_{0}^{2}+\mid c\mid^{2}}}$. Then, $Z\in U(1,1;\mathbb{H})$ and
\[
ZYXTX^{-1}Y^{-1}Z^{-1}=\begin{bmatrix}
d_{0}+\mathbf{i}\sqrt{1-d_{0}^{2}} & 0 \\
0 & d_{0}-\mathbf{i}\sqrt{1-d_{0}^{2}}
\end{bmatrix}.
\]

\noindent {\bf Case $3$:} $b\neq \overline{c}\neq 0$ and $\triangle(T)<0$.

In this case,
\[
\sigma_{R}(T)=\{ s\in\mathbb{H}:s_{0}=\frac{1}{2}((a_{0}+d_{0})\pm\sqrt{-\triangle(T)}),\ \mid s\mid^{2}=1 \}.
\]
Choose
\[
s=s_{0}+\mathbf{i}\sqrt{1-s_{0}^{2}} \ \ \text{and} \ \ s'=s_{0}^{\prime}+\mathbf{i}\sqrt{1-s_{0}^{\prime2}}
\]
as the representation elements of the two connected components of $\sigma_{R}(T)$ respectively, where $Re(s)=s_{0}=\frac{1}{2}(a_{0}+d_{0}+\sqrt{-\triangle(T)})$ and $Re(s')=s_{0}^{\prime}=\frac{1}{2}(a_{0}+d_{0}-\sqrt{-\triangle(T)})$.

Let
\[
T^{\prime}=\left\{\begin{array}{cc}
\begin{bmatrix}
s & 0 \\
0 & s'
\end{bmatrix} & \text{if} \ a_0>d_0, \\
\begin{bmatrix}
s' & 0 \\
0 & s
\end{bmatrix} & \text{if} \ a_0<d_0.
\end{array}\right.
\]
Now suppose that $a_0>d_0$. We aim to prove that there exists a matrix
\[
X=\begin{bmatrix}
x_{1} & x_{2} \\
x_{3} & x_{4}
\end{bmatrix}\in U(1,1;\mathbb{H})
\]
such that $XT=T^{\prime}X$, that is, there exist $x_{n},\ n=1,2,3,4$ satisfying (\ref{1.1}) and the following equations

\begin{numcases}{}
x_{1}a+x_{2}c=s x_{1} \label{3.1}\\
x_{1}b+x_{2}d=s x_{2} \label{3.2}\\
x_{3}a+x_{4}c=s'x_{3} \label{3.3}\\
x_{3}b+x_{4}d=s'x_{4} \label{3.4}
\end{numcases}

Next, we will find the solution of those equations.

It follows from equation (\ref{3.1}) that
\[
x_2=(sx_1-x_1a)\cdot c^{-1}.
\]
Substitute it into equation (\ref{3.2}), then
\[
x_1b+(sx_1-x_1a)c^{-1}d=s(sx_1-x_1a) c^{-1},
\]
and consequently,
\[
x_1b|c|^2+(sx_1-x_1a)\overline{c}d=(s^2x_1-sx_1a) \overline{c}.
\]
According to $T\in U(1,1;\mathbb{H})$ and $\overline{s}s=|s|^2=1$, we have
\begin{align*}
0=&s^2x_1\overline{c}-sx_1a\overline{c}-x_1b|c|^2-sx_1\overline{c}d+x_1a\overline{c}d \\
=& s^2x_1\overline{c}-sx_1a\overline{c}-x_1b|c|^2-sx_1\overline{c}d+x_1b\overline{d}d \\
=& s^2x_1\overline{c}-sx_1a\overline{c}-sx_1\overline{c}d+x_1b \\
=& s^2x_1\overline{c}-sx_1b\overline{d}-sx_1\overline{c}d+x_1b-x_1\overline{c}+\overline{s}sx_1\overline{c} \\
=&(s+\overline{s})sx_1\overline{c}-sx_1b\overline{d}-sx_1\overline{c}d+x_1(b-\overline{c}) \\
=& sx_1(2s_0\overline{c}-b\overline{d}-\overline{c}d)+x_1(b-\overline{c}).
\end{align*}
For convenience, let $p=2s_{0}\overline{c}-b\overline{d}-\overline{c}d$. Notice that

\begin{align*}
\mid p\mid^{2}-\mid b-\overline{c}\mid^{2}=& 4s_{0}^{2}\mid b\mid^{2}+2\mid b\mid^{2}\mid d\mid^{2}-4s_{0}\cdot Re(\overline{c}d\overline{b})-4s_{0}\cdot Re(\overline{c}\overline{d}c) \\
&+2Re(b\overline{d}\cdot\overline{d}c)-2\mid b\mid^{2}+2Re(bc) \\
=&(a_{0}+d_{0}+\sqrt{-\triangle(T)})^{2}\mid b\mid^{2}+2\mid b\mid^{2}\mid d\mid^{2}-4s_{0}\mid b\mid^{2}a_{0} \\
&-4s_{0}\mid b\mid^{2}d_{0}+2Re(b\overline{d}^{2}c)-2\mid b\mid^{2}+2Re(bc)   \\
=&(a_{0}+d_{0})^{2}\mid b\mid^{2}+2(a_{0}+d_{0})\sqrt{-\triangle(T)}\mid b\mid^{2}-\triangle(T)\mid b\mid^{2} \\
&+2\mid b\mid^{4}-4s_{0}\mid b\mid^{2}(a_{0}+d_{0})+2Re(b\overline{d}^{2}c)+2Re(bc)   \\
=&(a_{0}+d_{0})^{2}\mid b\mid^{2}+2(a_{0}+d_{0})(2s_0-(a_{0}+d_{0}))\mid b\mid^{2} \\
&-(\mid b-\overline{c}\mid^{2}-(a_{0}-d_{0})^{2})\mid b\mid^{2}+2\mid b\mid^{4} \\
&-4s_{0}\mid b\mid^{2}(a_{0}+d_{0})+2Re(b\overline{d}^{2}c)+2Re(bc)   \\
=&  -4a_{0}d_{0}\mid b\mid^{2}+\mid b\mid^{2}(2\mid b\mid^{2}-\mid b-\overline{c}\mid^{2})+2Re(bc) \\
&+2Re(b\overline{d}^{2}c) \\
=&  -4a_{0}d_{0}\mid b\mid^{2}+2(\mid b\mid^{2}+1)Re(bc)+2Re(b\overline{d}^{2}c) \\
=& -4a_{0}d_{0}\mid b\mid^{2}+2Re(b\overline{d}\cdot d \cdot c+b\overline{d}\cdot\overline{d}\cdot c)   \\
=&  -4a_{0}d_{0}\mid b\mid^{2}+4d_{0}\cdot Re(a\overline{c}c)   \\
=&  0.
\end{align*}
Following from $sx_1p+x_1(b-\overline{c})=0$, we have
\begin{equation}\label{x_{1}}
\mid b-\overline{c}\mid^{2}sx_{1}=-x_{1}(b-\overline{c})\overline{p}.
\end{equation}
Since
$$
\mid \mid b-\overline{c}\mid^{2}s\mid=\mid-(b-\overline{c})\overline{p}\mid \ \  \text{and} \ \ Re(\mid b-\overline{c}\mid^{2}s)=Re(-(b-\overline{c})\overline{p})=s_{0}\mid b-\overline{c}\mid^{2},
$$
there exists a nonzero quaternion $x_{1}$ such that the equation (\ref{x_{1}}) holds. Furthermore, one can obtain
\begin{equation}\label{x_{2}=}
x_{2}=(sx_1-x_1a)\cdot c^{-1}=-x_{1}((b-\overline{c})p^{-1}+a)c^{-1}.
\end{equation}

{\bf Claim A.} If $a_0>d_0$, then $|((b-\overline{c})p^{-1}+a)c^{-1}|<1$.

It suffices to prove
\[
|(b-\overline{c})+ap|^2<|cp|^2.
\]
Since $b\neq\overline{c}$,
\[
(d_{0}-a_{0})^2+\triangle(T)=\mid b-\overline{c}\mid^2>0.
\]
Consequently, by $a_0>d_0$ and $\triangle(T)<0$, we have
\[
(d_{0}-a_{0})+\sqrt{-\triangle(T)}<0.
\]
Notice that
\begin{align}
&Re((b-\overline{c})\bar{p}\bar{a}) \nonumber\\
=&Re(2s_{0}bc\overline{a}-bd\overline{b}\overline{a}-\mid b\mid^{2}\mid d\mid^{2}-2s_{0}\mid c\mid^{2}\overline{a}+\mid b\mid^{2}\overline{a}^{2}+\mid d\mid^{2}\overline{c}\overline{b}) \nonumber\\
=&Re(2s_{0}\overline{a}bc-\overline{a}bd\overline{b}-\mid b\mid^{2}\mid d\mid^{2}-2s_{0}\mid c\mid^{2}\overline{a}+\mid b\mid^{2}\overline{a}^{2}+\mid d\mid^{2}\overline{c}\overline{b})\nonumber\\
=&Re(2s_{0}\overline{c}dc)-Re(\overline{c}dd\overline{b}+\overline{c}d\overline{d}\overline{b})+Re(\overline{c}d\overline{d}\cdot\overline{b}+\mid d\mid^{2}\overline{c}\overline{b})-\mid b\mid^{2}\mid d\mid^{2} \nonumber\\
&-Re(2s_{0}\mid c\mid^{2}\overline{a})+Re(\mid b\mid^{2}\overline{a}^{2})\nonumber\\
=&2s_{0}\mid c\mid^{2}d_{0}-Re(\overline{c}d(d+\overline{d})\overline{b})+2Re(\mid d\mid^{2}\overline{c}\overline{b})-\mid b\mid^{2}\mid d\mid^{2}-2s_{0}\mid c\mid^{2}a_{0}\nonumber\\
&+\mid b\mid^{2}(2a_{0}^{2}-\mid a\mid^{2})\nonumber\\
=&2s_{0}\mid b\mid^{2}d_{0}-2s_{0}\mid b\mid^{2}a_{0}-2\mid b\mid^{2}a_{0}d_{0}+2\mid b\mid^{2}a_{0}^{2}-2\mid b\mid^{2}\mid d\mid^{2} \nonumber\\
&+2\mid d\mid^{2}Re(\overline{c}\overline{b}) \nonumber\\
=&2\mid b\mid^{2}(s_{0}-a_{0})(d_{0}-a_{0})-\mid d\mid^{2}(2\mid b\mid^{2}-2Re(\overline{c}\overline{b})) \nonumber\\
=&\mid b\mid^{2}((d_{0}-a_{0})^2+(d_{0}-a_{0})\sqrt{-\triangle(T)})-\mid d\mid^{2}\mid b-\overline{c}\mid^2 \nonumber\\
=&\frac{1}{2}\mid b\mid^{2}((d_{0}-a_{0})^2+2(d_{0}-a_{0})\sqrt{-\triangle(T)}+(d_{0}-a_{0})^2-\mid b-\overline{c}\mid^2) \nonumber\\
&+\frac{1}{2}\mid b\mid^{2}\mid b-\overline{c}\mid^2-\mid d\mid^{2}\mid b-\overline{c}\mid^2 \nonumber\\
=&\frac{1}{2}\mid b\mid^{2}((d_{0}-a_{0})+\sqrt{-\triangle(T)})^2-(\frac{1}{2}\mid b\mid^{2}+1)\mid b-\overline{c}\mid^2. \label{3.7}
\end{align}
Then,
\begin{align}
&|(b-\overline{c})+ap|^2-|cp|^2 \nonumber\\
=&|b-\overline{c}|^2+|ap|^2+2Re((b-\overline{c})\overline{p}\overline{a})-|cp|^2 \nonumber\\
=&2(Re((b-\overline{c})\overline{p}\overline{a})+|b-\overline{c}|^2) \nonumber\\
=&\mid b\mid^{2}\left(((d_{0}-a_{0})+\sqrt{-\triangle(T)})^2-\mid b-\overline{c}\mid^2\right)\nonumber\\
=&2\mid b\mid^{2}\left((d_{0}-a_{0})\sqrt{-\triangle(T)}-\triangle(T)\right)\label{3.8} \\
<&0.\nonumber
\end{align}
This proves the Claim A.

Notice that if a nonzero quaternion $x_{1}$ satisfies the equation (\ref{x_{1}}), then for any $t\in\mathbb{R}\setminus\{0\}$, $tx_{1}$ is also a nonzero quaternion satisfying the equation (\ref{x_{1}}). So, together with (\ref{x_{2}=}) and Claim A, we could choose $x_{1}$ such that not only the equation (\ref{x_{1}}) holds, but also $|x_1|^2-|x_2|^2=1$.

Similarly, let $p'=2s_{0}^{\prime}\overline{c}-b\overline{d}-\overline{c}d$, one can also see that
\[
\mid p'\mid^{2}=\mid b-\overline{c}\mid^{2}
\]
We could obtain a nonzero quaternion $x_{3}$ such that the following equation holds,
\begin{equation}\label{x_{3}}
\mid b-\overline{c}\mid^{2}s_{2}x_{3}=-x_{3}(b-\overline{c})\overline{p'}.
\end{equation}
Moreover, we put
\begin{equation}\label{x_{4}=}
 x_{4}=-x_{3}((b-\overline{c})(p')^{-1}+a)c^{-1}.
\end{equation}
Similar to the previous discussion, we could obtain

{\bf Claim B.} If $a_0>d_0$, then $|((b-\overline{c})(p')^{-1}+a)c^{-1}|>1$.

Then, together with (\ref{x_{4}=}) and Claim B, we could choose $x_{3}$ such that not only the equation (\ref{x_{3}}) holds, but also $|x_4|^2-|x_3|^2=1$.

Now we have obtained $x_{n},\ n=1,2,3,4$, which satisfies equations (\ref{3.1})-(\ref{3.4}), $|x_1|^2-|x_2|^2=1$ and $|x_4|^2-|x_3|^2=1$. To complete the proof, it remains to show $x_{1}\overline{x_{3}}=x_{2}\overline{x_{4}}$, $\overline{x_{1}}x_{2}=\overline{x_{3}}x_{4}$ and $|x_{1}|=|x_{4}|$. The following claim plays an important role in the remaining proof.

{\bf Claim C.}
\[
((b-\overline{c})p^{-1}+a)c^{-1}\cdot\overline{((b-\overline{c})(p')^{-1}+a)c^{-1}}=1.
\]

The Claim C is equivalent to
$$
((b-\overline{c})p^{-1}+a)\cdot (\overline{p'}^{-1}(\overline{b}-c)+\overline{a})=\mid c\mid^{2},
$$
that is,
\begin{equation}\label{gs1}
((b-\overline{c})\overline{p}+a|p|^2)(p'(\overline{b}-c)+\overline{a}|p'|^2)-|c|^2 \mid b-\overline{c}\mid^{4}=0.
\end{equation}

Since $p-p'=2(s_0-s'_0)\overline{c}=2\sqrt{-\triangle(T)}\overline{c}$, it follows from (\ref{3.8}) that
\begin{align}
&((b-\overline{c})\overline{p}+a|p|^2)(p'(\overline{b}-c)+\overline{a}|p'|^2)-|c|^2 \mid b-\overline{c}\mid^{4} \nonumber\\
=&((b-\overline{c})\overline{p}+a|p|^2)\overline{((b-\overline{c})\overline{p}+a|p|^2)}-|c|^2 \mid b-\overline{c}\mid^{4} \nonumber\\
 &-((b-\overline{c})\overline{p}+a|p|^2)2\sqrt{-\triangle(T)}\overline{c}(\overline{b}-c) \nonumber\\
=&2\mid b\mid^{2}\mid b-\overline{c}\mid^{2}\left((d_{0}-a_{0})\sqrt{-\triangle(T)}-\triangle(T)\right) \nonumber\\
&-2\sqrt{-\triangle(T)}((b-\overline{c})\overline{p}+a\mid b-\overline{c}\mid^{2})\overline{c}(\overline{b}-c). \label{3.9}
\end{align}
Notice that
\begin{align*}
&((b-\overline{c})\overline{p}+a\mid b-\overline{c}\mid^{2})\overline{c}(\overline{b}-c) \\
=&((b-\overline{c})(2s_0c-d\overline{b}-\overline{d}c)+a\mid b-\overline{c}\mid^{2})\overline{c}(\overline{b}-c) \\
=&(b-\overline{c})2s_0c\overline{c}(\overline{b}-c)-(b-\overline{c})d\overline{b}\overline{c}(\overline{b}-c) \\
 &-(b-\overline{c})\overline{d}c\overline{c}(\overline{b}-c)+(b-\overline{c})(\overline{b}-c)a\overline{c}(\overline{b}-c)\\
=&2s_0|c|^2\mid b-\overline{c}\mid^{2}-(b-\overline{c})c\bar{a}\bar{c}(\overline{b}-c) \\
&-|c|^2(b-\overline{c})\overline{d}(\overline{b}-c)+(b-\overline{c})(\overline{b}b\overline{d}-ca\overline{c})(\overline{b}-c)\\
=&2s_0|b|^2\mid b-\overline{c}\mid^{2}-2a_0|b|^2\mid b-\overline{c}\mid^{2} \\
=&|b|^2\mid b-\overline{c}\mid^{2}\left((d_0-a_0)+\sqrt{-\triangle(T)}\right).
\end{align*}
Then, together with (\ref{3.9}), we obtain
\[
((b-\overline{c})\overline{p}+a|p|^2)(p'(\overline{b}-c)+\overline{a}|p'|^2)-|c|^2 \mid b-\overline{c}\mid^{4}=0.
\]
This proves the Claim C.

{\bf (I) $|x_{1}|=|x_{4}|$.}

According to our selection of $x_{n},\ n=1,2,3,4$ ($|x_1|^2-|x_2|^2=1$, $|x_4|^2-|x_3|^2=1$ and equations (\ref{x_{2}=}) and (\ref{x_{4}=})), it suffices to prove
\[
|((b-\overline{c})p^{-1}+a)c^{-1}|\cdot|((b-\overline{c})(p')^{-1}+a)c^{-1}|=1,
\]
which is a corollary of Claim C.

{\bf (II) $x_{1}\overline{x_{3}}=x_{2}\overline{x_{4}}$.}

According to equations (\ref{x_{2}=}) and (\ref{x_{4}=}), it follows from Claim C that
\[
x_{2}\overline{x_{4}}=x_1\cdot(b-\overline{c})p^{-1}+a)c^{-1}\cdot\overline{((b-\overline{c})(p')^{-1}+a)c^{-1}}\cdot\overline{x_3}=x_{1}\overline{x_{3}}.
\]

{\bf (III) $\overline{x_{1}}x_{2}=\overline{x_{3}}x_{4}$.}

According to equations (\ref{x_{2}=}) and (\ref{x_{4}=}), one can see
\[
\overline{x_{1}}x_{2}=-|x_1|^2\cdot((b-\overline{c})p^{-1}+a)c^{-1}
\]
and
\[
\overline{x_{3}}x_{4}=-|x_4|^2\cdot\left(((b-\overline{c})(p')^{-1}+a)c^{-1}\right)^{-1}.
\]
Then, by Claim C and {\bf (I)} $|x_{1}|=|x_{4}|$, we have
\[
\overline{x_{1}}x_{2}=\overline{x_{3}}x_{4}.
\]

The proof for the case of $a_0<d_0$ is similar to the case of $a_0>d_0$.
\end{proof}

\begin{example}\label{duijiaohua}
Let $T=\begin{bmatrix}
2+\mathbf{i} & -\sqrt{2}+\sqrt{2}\mathbf{i} \\
-\sqrt{2}+\sqrt{2}\mathbf{i} & -1-2\mathbf{i}
\end{bmatrix}$.   Then $\triangle(T)=-1<0$ and $\sigma_{R}(T)=\{ s\in\mathbb{H}:s_{0}=0\ or\ 1,\ \mid s\mid=1\}$. Choose
$X=\begin{bmatrix}
\sqrt{2} & \mathbf{i} \\
-1 & -\sqrt{2}\mathbf{i}
\end{bmatrix}$. Then, one can see that
\[XTX^{-1}=\begin{bmatrix}
1 & 0 \\
0 & -\mathbf{i}
\end{bmatrix}.
\]
\end{example}

\begin{remark}
In the proof of the Case $3$ of Theorem \ref{M1}, we show that if $a_0>d_0$, $T$ is $U(1,1;\mathbb{H})$-similar to $\begin{bmatrix}
s & 0 \\
0 & s'
\end{bmatrix}$ in $U(1,1;\mathbb{H})$, and if $a_0<d_0$, $T$ is $U(1,1;\mathbb{H})$-similar to $\begin{bmatrix}
s' & 0 \\
0 & s
\end{bmatrix}$ in $U(1,1;\mathbb{H})$. It is worthy to notice that the two cases are not $U(1,1;\mathbb{H})$-similar, that is, $D_1=\begin{bmatrix}
s & 0 \\
0 & s'
\end{bmatrix}$ is not $U(1,1;\mathbb{H})$-similar to $D_2=\begin{bmatrix}
s' & 0 \\
0 & s
\end{bmatrix}$ in $U(1,1;\mathbb{H})$. Let $X=\begin{bmatrix}
x_1 & x_2 \\
x_3 & x_4
\end{bmatrix}$. If $XD_1=D_2X$, i.e,
\begin{equation}\label{end}
\begin{bmatrix}
x_1s & x_2s' \\
x_3s & x_4s'
\end{bmatrix}=\begin{bmatrix}
s'x_1 & s'x_2 \\
sx_3 & sx_4
\end{bmatrix}.
\end{equation}
Since $Re(s)\neq Re(s')$, $x_1s=s'x_1$ has only one solution $x_1=0$. Then the equation (\ref{end}) has no solution in $U(1,1;\mathbb{H})$. Therefore, $D_1$ is not $U(1,1;\mathbb{H})$-similar to $D_2$ in $U(1,1;\mathbb{H})$.
\end{remark}

\end{document}